\documentclass[12pt]{amsart}
\usepackage{amscd,amssymb,epsfig,mathtools}
\usepackage[usenames,dvipsnames]{color}
\calclayout

\numberwithin{equation}{section} 

\newcommand{\ud}{\,d} 
 
\newcommand{\R}{\mathbb{R}}

\renewcommand\P{{\mathcal P}}

\newcommand\cof{\operatorname{cof}}

\renewcommand{\div}{\operatorname{div}}

\newtheorem{thm}{Theorem}[section] 
 
\newtheorem{lemma}[thm]{Lemma}

\newtheorem{rem}[thm]{Remark}

\def\whsq{\vbox to 5.8pt 
{\offinterlineskip\hrule 
\hbox to 5.8pt{\vrule height 
5.1pt\hss\vrule height 5.1pt}\hrule}}

\def\<{\langle} 
\def\>{\rangle} 
\def\PP{{\mathop{{\rm I}\kern-.2em{\rm P}}\nolimits}} 
\def\FF{{\mathop{{\rm I}\kern-.2em{\rm F}}\nolimits}}   
\def\ZZ{{\mathop{{\rm I}\kern-.2em{\rm Z}}\nolimits}} 
 
\newlength{\sidemargin} 
\begin{document}
\title[]{
Pseudo transient continuation and time marching methods for Monge-Amp\`ere type equations}

\author{Gerard Awanou}
\address{Department of Mathematics, Statistics, and Computer Science, M/C 249.
University of Illinois at Chicago, 
Chicago, IL 60607-7045, USA}
\email{awanou@uic.edu}  
\urladdr{http://www.math.uic.edu/\~{}awanou}

\maketitle

\begin{abstract}
We present two numerical methods for the fully nonlinear elliptic Monge-Amp\`ere equation.
The first is a pseudo transient continuation method and the second is a pure pseudo time marching method. 
The methods are proven to converge to a strictly convex solution of a natural discrete variational formulation with $C^1$ conforming approximations. 
The assumption of existence of a strictly convex solution to the discrete problem is proven for smooth solutions of the continuous problem and supported by numerical evidence for non smooth solutions.
\end{abstract}

\section{Introduction}
We are interested in numerical solutions of the fully nonlinear elliptic Monge-Amp\`ere equation
\begin{equation}
\det  D^2 u = f  \ \text{in} \ \Omega, \quad u=g \ \text{on} \ \partial \Omega,
\label{m1}
\end{equation}
on a convex bounded domain $\Omega$ of $\R^n, n=2,3$ with boundary $\partial \Omega$. 
The unknown $u$
is a real valued function and 
 $f,g$ are given functions with $f>0$ in the non degenerate case and $f \geq 0$ in the degenerate case.
We will also assume that $f \in C(\Omega)$ and  $g$ 
in $C(\partial \Omega)$. 

Starting with \cite{Benamou2000,Dean2006}, interest has grown for finite element methods which are able to capture non smooth solutions
of second order fully nonlinear equations. For smooth solutions, the problem was studied in the context of semiconforming  $C^1$ finite elements by  B$\ddot{\text{o}}$hmer \cite{Bohmer2008,Bohmer2010} on both smooth and polygonal domains. B$\ddot{\text{o}}$hmer addressed general fully nonlinear elliptic equations for the first time. Brenner et al \cite{Brenner2010b} used Lagrange elements and interior penalty terms on smooth domains. Non smooth solutions can be handled with finite elements in the context of the vanishing moment methodology \cite{Feng2009}, a singular perturbation of \eqref{m1}.  
Proven convergence methods for non smooth solutions include the work of Oliker and Prussner \cite{Oliker1988}, Feng and Neilan \cite{Feng2012} for radial viscosity solutions in the finite element context and the work of Oberman \cite{Oberman2008} who addressed in general functions of the eigenvalues of the Hessian in the context of monotone finite difference methods.  For recent developments we refer to \cite{Feng2013}.

In this paper we give  numerical 
evidence that $C^1$ conforming approximations of a natural variational formulation of \eqref{m1} converge for non smooth solutions of the two dimensional problem. This is achieved by discretizing new iterative methods we introduce. We establish the convergence of the iterative methods under the assumption that the discrete problem has a strictly convex solution. We prove that such a solution exists  when \eqref{m1} has a smooth strictly convex solution. We do not assume that \eqref{m1} has a smooth solution for our iterative methods to converge. The existence of a convex solution to the discrete problem in the general case and the convergence of the discretization will be addressed in a subsequent paper. Even with such an existence result, it is still a non trivial task to solve the discrete nonlinear systems in situations where \eqref{m1} has a non smooth solution. This paper addresses this issue. 

The main technical innovation of this paper is the proof that in the context of $C^1$ conforming approximations, discrete  functions near a strictly convex solution are strictly convex. This explains why convexity did not need to be imposed explicitly in some previous studies. Newton's method remains the most appropriate iterative method for solving the discrete nonlinear equations when \eqref{m1} has a smooth convex solution. We give a new proof of convergence of Newton's method in the context of $C^1$ conforming approximations.

The results of this paper extend easily to finite dimensional spaces of piecewise smooth $C^1$ functions provided that the approximation property \eqref{schum} and inverse estimates \eqref{inverse} below hold.

Our results can be described in the general context of discretizations by $C^1$ elements of iterative methods for a general nonlinear elliptic equation $F(u)=0$. In the case of the Monge-Amp\`ere equation, $F(u) = \det D^2 u-f$. We first describe the iterative methods at the continuous level. However, we will not address convergence at the continuous level. As an initial guess we take the solution of the Poisson equation $\Delta u = n f^{1/n}$ in $\Omega, u=g$ on $\partial \Omega$.

\subsection{Pseudo transient continuation method}
We assume that $F$ differentiable and consider the sequence of problems
\begin{align*}
(\nu L+F'(u_k)) (u_{k+1}-u_k) &= -F(u_k), 
\end{align*}
where $L$ is a linear operator which can be taken as $L=-I $ where $I$ is the identity operator or $L=\Delta$ where $\Delta$ is the Laplace operator and $\nu >0$ is a parameter. 
Pseudo transient continuation methods \cite{Kelley98} form a general class of methods for solving nonlinear singular equations.
In the case of the Monge-Amp\`ere equation, the method consists in solving the sequence of approximate problems
\begin{align}\label{newtonseq}
\nu L \theta_k +(\text{cof} \ D^2 u_k) :D^2 \theta_k & = (f-f_k), \, f_k = \det D^2 u_k, \, \theta_k = u_{k+1}- u_k. 
\end{align}
Here $\cof A$ denotes the matrix of cofactors of the matrix $A$.

\subsection{Pseudo time marching method}
Given $\nu >0$, we consider the sequence of iterates
\begin{equation}
-\nu \Delta u_{k+1} = -\nu \Delta u_{k} + F(u_k), \ u_{k+1}=g \ \mathrm{on} \ \partial \Omega. \label{m2}
\end{equation}
This can be interpreted as an Euler discretization of the pseudo time dependent equation
$
\frac{\partial \Delta u}{\partial t} + F(u) = 0, 
$
or as a Laplacian preconditioner of a simple pseudo time marching algorithm, \cite{GlowinskiICIAM07} 
$
u_{k+1} = u_ k -\frac{1}{\nu} \Delta^{-1} F(u_k ).
$
See also a remark in \cite{Loeper2005}.

To the author's best knowledge, this is the first time the pseudo transient continuation method and the time marching method are used to indicate numerically convergence to viscosity solutions of finite element type methods for the Monge-Amp\`ere equation.

\subsection{Advantages and comparison of the two methods}

The methods we propose can be used in the context of different types of discretizations allowing us in particular to treat more easily non-rectangular domains. The methods can be accelerated with fast Poisson solvers and multigrid methods. This latter property is even more striking for the time marching method as its implementation requires only having access to a multigrid Poisson solver.


Although the theory of the Monge-Amp\`ere equation has concentrated on convex solutions, one can equally focus on concave solutions. We found out that \eqref{m2} is better able to capture concave solutions. It is easy to implement, requiring only a Poisson solver. For example one can capture weak solutions of the Monge-Ampere equation by simply discretizing \eqref{m2} with the standard Lagrange finite elements. The time marching method can also be applied to fully nonlinear equations such as the Pucci equation where $F$ is not differentiable. 
 
 In summary the pseudo transient continuation methods are better for smooth solutions and singular solutions on a coarse mesh. Otherwise the method of choice is the time marching method.

\subsection{Organization of the paper}

We organize the paper as follows: in the second section we introduce some notation and prove the key result that discrete  functions near a strictly convex solution are strictly convex. We introduce the natural variational formulation of \eqref{m1} and state an existence and uniqueness result for the discrete problem. 
As a corollary the discrete variational problem has a convex solution when \eqref{m1} has a smooth convex solution. 
We study the pseudo transient continuation methods in section \ref{continuation}. A special case is Newton's method for which we prove a quadratic convergence rate. The time marching methods are studied in section \ref{time}. The last section is devoted to numerical results. We give a brief description of the spline element method which is used for the computations  and offer heuristics about why our methods appear to preserve convexity.

\section{Notation and preliminaries}

We use the standard notation for the Sobolev spaces $W^{k,p}(\Omega)$ with norms $||.||_{k,p}$ and semi-norm $|.|_{k,p}$. In particular,
$H^k(\Omega)=W^{k,2}(\Omega)$ and in this case, the norm and semi-norms will be denoted respectively by 
$||.||_{k}$ and $|.|_{k}$. 
For a vector field $v=(v_i)_{i=1,\ldots,d}$ with values in $W^{k,p}(\Omega)^n, 1 \leq p < \infty$, we set $||v||_{k,p}= (\sum_{i=1}^n ||v_i||_{k,p}^2)^{\frac{1}{2}}$ and a similar notation for $|v|_{k,p}$. In the case $p=\infty$, we set
$||v||_{k,\infty} = \max_{i=1,\ldots,n} ||v_i||_{k,\infty}$ with a similar notation for $|v|_{k,\infty}$. For matrix valued fields, the above notation is extended canonically.

We make the usual convention of denoting  constants by $C$ but will occasionally index some constants. 

We make the assumption 
 that the boundary of $\Omega$ is polygonal and
that the triangulation $\mathcal{T}$ is shape regular in the sense that there is a constant $C>0$ such that for any triangle $K$, 
$h_K/\rho_K \leq C$, where $h_K$ denotes the diameter of $K$ and $\rho_K$ the radius of the largest ball contained in $K$. We also require the triangulation to be quasi-uniform in the sense that $h/h_{min}$ is bounded where $h$ and $h_{min}$ are the maximum and minimum respectively of $\{h_K, K \in \mathcal{T}_h \}$.

We define
\begin{equation}\label{sspaces}
V^h:=S^1_d( \mathcal{T} )=\{s \in C^1(\Omega), \ s|_t \in \P_d, \ \forall t \in \mathcal{T} \},
\end{equation}
where $\P_d$ denotes the space of polynomials of degree less than or equal to $d$.

In {\it two dimensions}, it is known that, \cite{Lai2007}, for $d \geq 5$ and $0 \leq l \leq d$, 
there exists a linear quasi-interpolation operator
$Q_h$ mapping $L_1(\Omega)$ into the spline space $S^1_d(\mathcal{T})$ and a constant $C$
such that if $v$ is in the Sobolev space $W^{l+1,p}(\Omega), 1 \leq p \leq \infty$
\begin{equation}
|| v -Q_h v ||_{k,p} \leq C_{} h^{l+1-k} |v|_{l+1,p}, \label{schum}
\end{equation}
for $0 \leq k \leq l$. If $\Omega$ is convex, the constant $C$ in \eqref{schum} depends only on $d,l$ and on the smallest
angle $\theta_{h}$ in $\mathcal{T}$. In the nonconvex case, $C$ depends only on 
the Lipschitz constant associated
with the boundary of $\Omega$. It is also known c.f. \cite{Dyer2003} that
the full approximation property for spline spaces holds on special triangulations for certain values of $d$.

In {\it three dimensions}, \eqref{schum} holds in general  for $d \geq 9$, c.f. \cite{Lai2007}. 

Note that, by \eqref{schum},
\begin{equation}
||Q_h v||_{2,p} \leq C_{} ||v||_{2,p}, \quad  v \in W^{2,p}(\Omega), \label{stable}
\end{equation}
for all $p \geq 1$. 

We assume that the following inverse inequality holds 
\begin{equation}
||v||_{s,p} \leq C_{} h^{l-s+\text{min}(0,\frac{n}{p}-\frac{n}{q})} ||v||_{l,q}, \forall v \in V^h, \label{inverse}
\end{equation}
 for $0 \leq l \leq s, 1 \leq p,q\leq \infty$. For $C^1$ finite element spaces the result can be found in \cite{Brenner02}, Theorem 4.5.11. For the spline spaces
they may be viewed as a consequence of the assumption of uniform triangulation and of Markov inequality, \cite{Lai2007} p. 2. See also \cite{Bohmer2010}, section 4.2.6. 



\subsection{Variational formulations}
We first recall the divergence form of the determinant and the expression of its Fr\'echet derivative.

For two $n \times n$ matrices $M, N$, we recall the Frobenius product 
$$M:N = \sum_{i,j=1}^n M_{ij} N_{ij}. $$ 
In particular, for a matrix $A$, we have $(\cof A):A=\sum_{i,j=1}^n (\cof A)_{ij} A_{ij}$.
\begin{lemma} \label{det-lem} We have
\begin{equation}\label{detform}
\det D^2 u = \frac{1}{n}(\cof D^2 u): D^2 u =
\frac{1}{n} \div\big((\cof D^2 u) D u \big). 
\end{equation}
And for $F(u) = \det D^2 u $ we have
$$
F'(u) (w) =  (\cof D^2 u): D^2 w = \div \big((\cof D^2 u) D w \big), 
$$
 for $u,w$ 
sufficiently smooth.
\end{lemma}
\begin{proof}
Note that for any $n \times n$ matrix $A$, $\det A =  (\cof A):A/n$, 
where $\cof A$ is the matrix of cofactors of $A$. This follows from the row expansion definition of the determinant.

For any sufficiently smooth matrix field $A$ and vector field $v$,
 $\div A^T v = (\div A) \cdot v + A:D v$. 
Here the divergence of a matrix field is the divergence operator applied row-wise.
 If we put $v= D u$, then 
$\det D^2 u=  (\cof D^2 u):(D^2 u)/n = (\cof D v): (D v)/n $ and
$\div(\cof D v)^T v = \div (\cof D v) \cdot v + (\cof D v): D v.$
But $\div \cof D v =0$, c.f. for example \cite{Evans1998} p. 440. Hence 
since $D^2 u$ and $\cof D^2 u$ are symmetric matrices \eqref{detform} follows. The assertion about the Fr\'echet derivative of $F$ follows from the definition of the determinant as a multilinear map (e.g \eqref{detform}) and the definition of matrix of cofactors. See also \cite{Evans1998} p. 440.

\end{proof}

Using the divergence form of the determinant \eqref{detform} and integration by parts, one obtains
the variational formulation of \eqref{m1}  given by: find $u \in W^{2,n}(\Omega)$, $u=g$ on $\partial \Omega$ such
that
\begin{equation}
-\frac{1}{n} \int_{\Omega} (\cof D^2 u) D u \cdot D w \ \ud x  =
 \int_{\Omega} f w \ \ud x , \quad \forall w \in W^{2,n}(\Omega) \cap H_0^1(\Omega) \label{var1}.
\end{equation}
We show that for $u \in W^{2,n}(\Omega)$, \eqref{var1}  is well defined. 

{\it Case $n=2$}. For $n=2$, each entry of $\cof D^2 u$ consists of a second derivative $\partial^2 u/(\partial x_i \partial x_j), i,j=1,\ldots,n$. By H$\ddot{\text{o}}$lder's inequality,
\begin{align*}
\bigg| \int_{\Omega} (\cof D^2 u) D u \cdot D w \ \ud x\bigg| \leq C ||D^2u||_{0,2} ||Du||_{0,4} ||Dw||_{0,4}.
\end{align*}
Next for $u \in H^2(\Omega), \partial u/\partial x_i \in H^1(\Omega), i=1,\ldots,n$ and by Sobolev embedding, i.e. the embedding of $H^1(\Omega)$
in $L^q(\Omega)$ for $q \geq 1$ when $n=2$, 
the right hand side above is bounded by 
$C ||D^2u||_{L^2(\Omega)} ||u||_{H^2(\Omega)}$ $ ||w||_{H^2(\Omega)}$.

{\it Case $n=3$}. For $n=3$, each entry of $\cof D^2 u$ involves the product of two second order derivatives. We have
by H$\ddot{\text{o}}$lder's inequality and Sobolev embedding, i.e.  the embedding of $H^1(\Omega)$
in $L^q(\Omega)$ for $1\leq q \leq 6$ when $n=3$, 
\begin{align*}
\bigg| \int_{\Omega} \frac{\partial^2 u}{\partial x_1^2} \frac{\partial^2 u}{\partial x_3^2} \frac{\partial u}{\partial x_1} 
\frac{\partial w}{\partial x_2} \ud x\bigg| & \leq || \frac{\partial^2 u}{\partial x_1^2} ||_{0,3} 
|| \frac{\partial^2 u}{\partial x_3^2} ||_{0,3}  || \frac{\partial u}{\partial x_1} ||_{0,6} 
|| \frac{\partial w}{\partial x_2} ||_{0,6} \\
& \leq ||u||_{2,3}^2 ||u||_2  ||w||_2.
\end{align*}
We conclude that for $n=3$, 
\begin{align*}
\bigg| \int_{\Omega} (\cof D^2 u) D u \cdot D w \, \ud x\bigg| \leq C ||u||_{2,3}^2 ||u||_2  ||w||_2.
\end{align*}
In summary for $n=2,3$, we may write
\begin{align}
\bigg| \int_{\Omega} (\cof D^2 u) D u \cdot D w \, \ud x \bigg| \leq C ||u||_{2,n}^{n-1} ||u||_2  ||w||_2. \label{detineq}
\end{align}

Put $V=W^{2,n}(\Omega)$ and $V_0 = W^{2,n}(\Omega) \cap H_0^1(\Omega)$. Note that $V^h$ given by \eqref{sspaces} satisfies 
$V^h \subset W^{2,n}(\Omega)$. Let  
$V^h_0 =V^h \cap H_0^1(\Omega)$ 
and furthermore let $g_h$ be the interpolant in $V^h$ of a smooth extension of $g$.

We have the following conforming discretization of \eqref{var1}:
find $u_h \in V^h$, $u_h=g_h$ on $\partial \Omega$ such
that
\begin{equation}
-\frac{1}{n}\int_{\Omega} (\cof D^2 u_h) D u_h \cdot D w_h \ \ud x  =
 \int_{\Omega} f w_h \ \ud x , \quad \forall w_h \in V_0^h. \label{var1h}
\end{equation}

We now present a number of preliminary results.
\subsection{Preliminary results}
We first prove that when \eqref{m1} has a smooth strictly convex solution, \eqref{var1h} has a unique local solution and we give error estimates. After introducing tools for computations with determinants, we show that a finite element function sufficiently close to a strictly convex finite element function is also strictly convex. It follows that the solution $u_h$ of \eqref{var1h} is strictly convex when \eqref{m1} has a smooth strictly convex solution.
\begin{thm} \label{errorest}
Let $3\leq l \leq d$ and assume that $u \in W^{l+1,\infty}(\Omega) $ is a strictly convex function, that $\Omega$ is convex with a polygonal boundary and that the spaces $V^h$ have the optimal approximation property \eqref{schum} and satisfy the inverse estimates \eqref{inverse}.
Then the problem \eqref{var1h} has a unique solution $u_h$ for $h$ sufficiently small and we have the error estimates
\begin{align*}
||u-u_h||_2 & \leq C h^{l-1} \\
 ||u-u_h||_1 & \leq C h^l \\
||u-u_h||_{0} & \leq  C h^{l+1}  + C( h^{l-1-\frac{n}{2} } +  C  )^{n-2}h^{2l-1-\frac{n}{2} },
\end{align*} 
with a constant $C$ which depends on $u$ but is independent of $h$.
\end{thm}

\begin{proof}
The $H^1$ error estimate is given in \cite{Bohmer2008}, Theorems 5.1 and 8.7. See also \cite{Bohmer2010}. The $H^2$ error estimate follows from an inverse estimate. For the proof of the $L^2$ error estimate, the proof in \cite{Feng2009} can be adapted. The results of \cite{Brenner2010b,Brenner2010a} also give the error estimates in the theorem. They were given with the interior penalty formulation but the variational problems discussed there reduce to the one considered in this paper for $C^1$ finite element spaces. For another proof of the $H^1$ error estimate, we refer to \cite{Awanou-Std01}. 
\end{proof}

Next we give some preliminary results which are essential for computations with terms involving the determinant. 

We first recall the Mean Value Theorem for Banach spaces. Let $E$ and $F$ be Banach spaces 
and let us denote by $L(E,F)$ the space of continuous linear mappings from $E$ to $F$. Let also 
$X$ be an open subset of $E$ and let $F: X \to F$ be a differentiable map. If 
$F': X \to L(E,F)$ is continuous, $F$ is said to be of class $C^1$ and for all $a,x \in X$, we have
$$
F(x) = F(a) + \int_0^1 F'[(1-t) a + t x](x-a) \ud t.
$$

\begin{lemma} \label{cof-bound} 
For $n=2$ and $n=3$, and two matrix fields $\eta$ and $\tau$
$$||\cof (\eta) - \cof (\tau) ||_{\infty} \leq (n-1)^2 (||\eta||_{\infty} + ||\tau||_{\infty})^{n-2}||\eta-\tau||_{\infty}.$$
\end{lemma}
\begin{proof}
For $n=2$, we have $\cof (\eta) - \cof (\tau)=  \cof (\eta-\tau)$ from which the result follows. For $n=3$ we use the Mean Value Theorem. It is enough to estimate the first entry of 
$\cof (\eta) - \cof (\tau)$ which is equal to
\begin{align*}
\det \begin{pmatrix} \eta_{22} & \eta_{23} \\ \eta_{32} & \eta_{33} \end{pmatrix} - \det \begin{pmatrix} \tau_{22} & \tau_{23} \\ \tau_{32} & \tau_{33} \end{pmatrix} & =
\cof \bigg( t \begin{pmatrix} \eta_{22} & \eta_{23} \\ \eta_{32} & \eta_{33} \end{pmatrix}  +(1-t) \begin{pmatrix} \tau_{22} & \tau_{23} \\ \tau_{32} & \tau_{33} \end{pmatrix} \bigg): \\
& \begin{pmatrix} \eta_{22} - \tau_{22} &\eta_{23} - \tau_{23} \\\eta_{32} - \tau_{32} & \eta_{33} - \tau_{33} \end{pmatrix}, 
\end{align*}
for some $t \in [0,1]$. The result then follows.
\end{proof}

\begin{lemma} \label{ineqlem}
Let $v, w \in W^{2,n}(\Omega), n=2,3$ and $\psi \in H_0^1(\Omega) \cap H^2(\Omega)$, then 
\begin{align*}
\begin{split}
\int_{\Omega} (\det D^2 v - \det D^2 w) \psi \ud x & = 
 -\int_0^1 \bigg\{ \int_{\Omega}
\big((\cof [ (1-t) D^2 w + t D^2 v)]  \\
& \qquad  \qquad \qquad (D v - D w) \big) \cdot D \psi \ud x \bigg\} \ud t, 
\end{split}
\end{align*}
and
if in addition $v,w \in  W^{2,n}(\Omega) \cap W^{2,\infty}(\Omega)$
\begin{align}
\begin{split} \label{meanv1}
\bigg| \int_{\Omega} (\det D^2 v - \det D^2 w) \psi \ud x \bigg| & \leq n (|v|_{2,\infty}+|w|_{2,\infty})^{n-1} |v-w|_1  |\psi |_1,
\end{split}
\end{align}
and
\begin{align}
\begin{split} \label{meanv2}
\bigg| \int_{\Omega} [( \cof D^2 v  - \cof D^2 w) & D (v-w)] \cdot D \psi   \ud x \bigg|   \leq n (n-1)^2 (|v|_{2,\infty}  \\
& \quad \quad \quad +|w|_{2,\infty})^{n-2}   |v-w|_{2,\infty} |v-w|_1  |\psi |_1.
\end{split}
\end{align}
\end{lemma}

\begin{proof}
We first note that for a matrix field $A$ and vector fields $b, c$, we have
$
(A b) \cdot c  = \sum_{i=1}^n (A b)_i c_i = \sum_{i,j=1}^n A_{ij} b_j c_i.
$ 
Thus by Cauchy-Schwarz inequality,
\begin{align*}
\int_{\Omega} (A b) \cdot c   & \leq ||A||_{\infty}  \sum_{i,j=1}^n \int_{\Omega} |b_j c_i| \leq ||A||_{\infty}  \sum_{i,j=1}^n ||b_j ||_0 || c_i||_0 \\
& =  ||A||_{\infty} \bigg(  \sum_{i=1}^n  || c_i||_0 \bigg) \bigg( 
\sum_{j=1}^n ||b_j ||_0
\bigg) \\
&  \leq n ||A||_{\infty}  \bigg( \sum_{i=1}^n  || c_i||_0^2 \bigg)^{\frac{1}{2}}
\bigg( \sum_{j=1}^n ||b_j ||_0^2 \bigg)^{\frac{1}{2}}.
\end{align*}
It follows that for 
$L^{\infty}$  valued matrix fields $A, B$ and $v, w \in W^{2,n}(\Omega), n=2,3$, by Cauchy-Schwarz inequality
\begin{align} \label{detineq01}
\bigg| \int_{\Omega} [(\cof A - \cof B) D v] D w \ud x
\bigg| \leq n ||\cof A - \cof B||_{\infty}^{} |v |_1  |w |_1.
\end{align}
Next, let $F: C^{\infty}(\Omega) \to C^{\infty}(\Omega)$ denote the mapping $v \mapsto \det D^2 v$. Then $F$ is differentiable with
$$
F'[u](v) = (\cof D^2 u) :D^2 v = \div\big((\cof D^2 u) D v \big).
$$
Since $v \mapsto F'[v] $ is linear, $F$ is of class $C^1$ and by the Mean Value Theorem
$$
F(v) - F(w) = \int_0^1 \div\big((\cof(1-t) D^2 w + t D^2 v) (D v - D w) \big) \ud t.
$$
It follows that for $\psi \in \mathcal{D}(\Omega)$, and $v, w \in C^{\infty}(\Omega) \cap W^{2,\infty}(\Omega)$,
$$
\int_{\Omega} (\det D^2 v - \det D^2 w) \psi \ud x = \int_{\Omega}  \bigg\{ \int_0^1 \div\big((\cof(1-t) D^2 w + t D^2 v) (D v - D w) \big) \ud t \bigg\} \psi \ud x.
$$
By Fubini's theorem,
\begin{align*}
\begin{split}
\int_{\Omega} (\det D^2 v - \det D^2 w) \psi \ud x & = \int_0^1 \bigg\{ \int_{\Omega}
 \div\big((\cof(1-t) D^2 w + t D^2 v) (D v - D w) \big) \psi \ud x \bigg\} \ud t \\
=&   -\int_0^1 \bigg\{ \int_{\Omega}
[ (\cof(1-t) D^2 w + t D^2 v) (D v - D w) ] \cdot D \psi \ud x \bigg\} \ud t. 
\end{split}
\end{align*}
Applying  \eqref{detineq01}, we obtain
\begin{align*}
\begin{split}
\bigg|\int_{\Omega} (\det D^2 v - \det D^2 w) \psi \ud x \bigg| & \leq n \int_0^1 ||\cof (1-t) D^2 w + t D^2 v||_{\infty}  |v-w|_1  |\psi |_1 \ud t \\
& \leq n  \int_0^1 ||(1-t) D^2 w + t D^2 v||_{\infty}^{n-1}  |v-w|_1  |\psi |_1 \ud t \\
& \leq n (|v|_{2,\infty}+|w|_{2,\infty})^{n-1} |v-w|_1  |\psi |_1.
\end{split}
\end{align*}
We have therefore obtained \eqref{meanv1} for $v, w \in C^{\infty}(\Omega) \cap W^{2,\infty}(\Omega)$ and $\psi \in \mathcal{D}(\Omega)$. We recall that $\mathcal{D}(\Omega)$ is dense in
$H_0^1(\Omega) \cap H^2(\Omega)$ and $C^{\infty}(\Omega) \cap W^{2,\infty}(\Omega)$ is dense in $W^{2,\infty}(\Omega)$. We then obtain \eqref{meanv1} by a density argument.

Inequality \eqref{meanv2} is a direct consequence of \eqref{detineq01} and  Lemma \ref{cof-bound}. 
\end{proof}

Let $\lambda_1(A)$ and $\lambda_n(A)$ denote the smallest and largest eigenvalues of a symmetric matrix $A$. Since $\det D^2 u \geq f \geq c_0 >0$ and $u$ is smooth and convex, there exist constants $m', M' >0$, independent of $h$
\begin{equation} \label{psc}
m' \leq \lambda_1(D^2 u(x)) \leq \lambda_n(D^2 u(x))  \leq M', \forall x \in \Omega.
\end{equation}
It follows from \cite{Hoffman53} Theorem 1 and Remark 2 p. 39 that for two symmetric $n \times n$ matrices $A$ and $B$, 
\begin{equation} \label{cont-eig}
|\lambda_k(A) - \lambda_k(B)| \leq n \max_{i,j} |A_{ij} - B_{ij}|, k=1, \ldots, n.
\end{equation}
It follows that for $u, v \in W^{2,\infty}(\Omega)$, 
\begin{align}
|\lambda_1( D^2 u(x)) - \lambda_1( D^2 v(x))| & \leq n |u-v|_{2,\infty} \label{lambda1}\\
|\lambda_n( D^2 u(x)) - \lambda_n( D^2 v(x))| & \leq n |u-v|_{2,\infty} \label{lambdan}.
\end{align}

By \eqref{inverse} we have for $v \in V^h$
$$
|v|_{2,\infty} \leq C_0 h^{-1-\frac{n}{2}} ||v||_1.
$$
Let $\delta >0$ such that
\begin{equation} \label{delta}
\delta < \min \bigg\{ \, 1, \frac{m'}{2 n C_{0}} \, \bigg\}. 
\end{equation}
\begin{lemma} \label{lem0}
For $h$ sufficiently small and for all $v_h \in V^h$ with $||v_h- Q_h u||_1 < \delta h^{1+n/2}/2$, $D^2 (v_h|_K)$ is positive definite with 
$$\frac{m'}{2} \leq \lambda_1D^2 (v_h|_K) \leq \lambda_nD^2 (v_h|_K) \leq \frac{3 M'}{2},$$ 
where $m'$ and $M'$ are the constants of Assumption \eqref{psc}. It follows that $v_h$ is convex. 
\end{lemma}

\begin{proof} 
For $v \in W^{2,\infty}(\Omega)$, $|v-u|_{2,\infty} \leq \delta C_{0}$ and \eqref{delta}
imply
\begin{align*}
|\lambda_1( D^2 v(x)) - \lambda_1( D^2 u(x))| & \leq n |v-u|_{2,\infty} \leq n \delta C_{0} \leq \frac{m'}{2} \, \text{ a.e. in} \,  \Omega, 
\end{align*}
since $ \delta < m'/(2 n C_0$.
By Assumption \eqref{psc} $\lambda_1(  D^2 u(x)) \geq  m',$ and thus
$\lambda_1(D^2 v(x)) \geq \lambda_1(D^2 u(x)) - m'/2 \geq m'/2$ a.e. in   $ \Omega$. We conclude that for $|v-u|_{2,\infty} \leq \delta C_{0}$, $\lambda_1( D^2 v(x)) > m'/2, $ a.e. in   $ \Omega$. 

Now, by \eqref{schum}, $|u-Q_h u|_{2,\infty} \leq C_{} h^{d-1} |u|_{d+1,\infty}$. So for $h$ sufficiently small,  $|u-Q_h u|_{2,\infty} \leq \delta C_{0}/2$. Moreover by \eqref{inverse} and the assumption of the lemma
\begin{align*}
|v_h-Q_h u|_{2,\infty} \leq C_{0} h^{-1-\frac{n}{2}} ||v_h-Q_h u||_1 \leq \frac{\delta C_{0}}{2}.
\end{align*}
Therefore $|v_h-u|_{2,\infty} \leq \delta C_{0}$ as well and it follows that $\lambda_1( D^2 v_h(x)) > m'/2, $ a.e. in   $ \Omega$ as claimed. 

Since $m' \leq M'$, we also have  $|\lambda_n( D^2 v_h(x)) - \lambda_n( D^2 u(x))| \leq M'/2$ a.e. in   $ \Omega$. Thus
$\lambda_n( D^2 v_h(x)) \leq \lambda_n( D^2 u(x)) +M'/2 \leq 3 M'/2$. 

Since $v_h$ is piecewise convex and $C^1$, $v_h$ is convex \cite{Lai2000} Lemma 1. This concludes the proof.
\end{proof}
Put
$$X^h= \{ \, v_h \in V^h, v_h=g_h \, \text{on} \, \partial \Omega, ||v_h-Q_h u||_1 < \frac{\delta h^{1+\frac{n}{2}}}{4} \, \}.$$
By Lemma \ref{lem0}, for $h$ sufficiently small and $v_h \in X^h, ||v_h- Q_h u||_1 < \delta h^{1+n/2}/2$ and hence $v_h$ is  convex with smallest eigenvalue bounded a.e. below by $m'/2$ and above by $3M'/2$. 

As a consequence of Assumption \eqref{psc} we have
\begin{lemma} \label{lem-1}
For $h$ sufficiently small and all $v_h \in X_h$
$$
m \leq \lambda_1(\cof D^2 v_h(x)) \leq \lambda_n(\cof D^2 v_h(x))  \leq M, \forall x \in K, K \in \mathcal{T}_h,
$$ 
with $m=(m')^{n}/M'$ and $M=(M')^{n}/m$.

It follows that for $ w \in H^1(K)$
\begin{equation} \label{p-d-K}
m |w|_{1,K}^2 \leq \int_{K} [(\cof\, D^2 v_h(x)) D w(x)] \cdot D w(x) \, \ud x \leq  M |w|_{1,K}^2. 
\end{equation}
\end{lemma}
\begin{proof} 
We first note that by Lemma \ref{lem0}, there exist constants $m, M >0$ such that $
m \leq \lambda_1(\cof D^2 v_h(x)) \leq \lambda_n(\cof D^2 v_h(x))  \leq M $ a.e. in $\Omega$ for $v_h \in X_h$. 
To prove this, recall that for an invertible  matrix $A$,
$\cof A = (\det A) (A^{-1})^T$. Since a matrix and its transpose have the same set of eigenvalues, the eigenvalues of $\cof A$ are of the form $ \det A/\lambda_i$ where 
$\lambda_i, i=1,\ldots,n$ is an eigenvalue of $A$. Applying this observation to $A = D^2 u(x)$ and using Lemma \ref{lem0}, we obtain that the eigenvalues of
$\cof D^2 v_h(x)$ are a.e. uniformly bounded below by $m=(m')^{n}/M'$ and above by $M=(M')^{n}/m$.

Since $\lambda_1(D^2 v_h(x))$ and $\lambda_n(D^2 v_h(x))$ are the minimum
and maximum respectively of the Rayleigh quotient $[(\cof\, D^2 v_h(x)) z] \cdot z/||z||^2$, where $||z||$ denotes the standard Euclidean norm in $\R^n$, we have
$$
m ||z||^2 \leq [(\cof\, D^2 v_h(x)) z] \cdot z \leq  M ||z||^2, z \in \R^n. 
$$
This implies
\begin{equation*}
m |w|_{1,K}^2 \leq \int_{K} [(\cof\, D^2 v_h(x)) D w(x)] \cdot D w(x) \, \ud x \leq  M |w|_{1,K}^2, w \in H^1(K). 
\end{equation*}

\end{proof}

\begin{rem} \label{conv-disc-sol}
As a consequence of Lemma \ref{lem0} and Theorem \ref{errorest}, for $h$ sufficiently small, the solution $u_h$ of \eqref{var1h} is in $X^h$ and hence convex.
\end{rem}

Let $(V_0^h)'$ denote the dual space of
$V_0^h$ with $V_0^h$ equipped with $||.||_1$. We consider the mapping $F_h: V^h \to (V_0^h)'$  defined by 
$$\<F_h(v_h),\psi_h\> = \int_{\Omega} (\det D^2 v_h)  \psi_h \ud x, v_h \in V^h, \psi_h \in V_0^h,$$ and recall that $f > 0$ is continuous. Since $\Omega$
is bounded, by $L^2$ duality, $V^h \subset (V_0^h)'$. We use the notation $||.||$ for the operator norm of an element of a dual space. 
With this notation, \eqref{var1} can be written $\det D^2 u =f$ in $V_0'$ and
\eqref{var1h} can be written $F_h(u_h)=f$ in $(V_0^h)'$.

Let $C_1$ denote the constant in the Poincare's inequality, i.e. $||p||_1 \leq C_1 |p|_1, p \in H_0^1(\Omega)$. Without loss of generality, we may assume that $C_1 \leq 1$, for example by assuming that 
the domain $\Omega$ is contained in a cube of side length at most 1, \cite{Braess} p.30. 

Note that by Lemma \ref{det-lem} and integration by parts, 
\begin{equation} \label{Frechet-h} 
\<F_h'(v_h)(p), p\> = -\int_{\Omega}[(\cof D^2 v_h) D p] \cdot D p \ud x, p \in H_0^1(\Omega).
\end{equation}
Thus by Lemma \ref{lem-1} and $h$ sufficiently small 
\begin{equation}
-M ||p||_1^2  \leq \<F_h'(v_h)(p), p\> \leq -\frac{m}{C_1^2} ||p||_1^2, p \in H_0^1(\Omega), \label{fundpos}
\end{equation}
for $v_h \in X^h$. 
We have
\begin{lemma} \label{prop-Fh} The following properties hold for $h$ sufficiently small.

{\it Discrete coercivity}: 
\begin{equation*} 
||F_h'(v_h)(p)|| \geq \frac{m}{C_1^2}   ||p||_1, \forall  p \in V_0^h \, \text{and} \, v_h \in X^h. 
\end{equation*}

{\it Generalized Lipschitz continuity}: 
\begin{equation} \label{glc}
||F_h'(v_h)(\psi) - F_h'(w_h)(\psi)|| \leq C_2 h^{-1-\frac{n}{2}}  ||v_h-w_h||_{1}^{} ||\psi||_{1},
\end{equation}
for $v_h,w_h \in X^h$ and $\psi \in V^h, \eta \in V_0^h$ and with $C_2= C || u||_{2,\infty}.$
\end{lemma}

\begin{proof}

By \eqref{fundpos} $||F_h'(v)(p)|| = \text{sup}_{\psi \neq 0} |\<F'_h(v)(p), \psi\>|/||\psi||_1 \geq m/C_1^2  ||p||_1$, which proves the discrete coercivity condition.

For $v_h,w_h \in X^h$, $\psi \in V^h$, $\eta \in V_0^h$, we have
\begin{align} \label{F-h-temp}
\begin{split}
\<F_h'(v_h)(\psi),\eta\> & - \<F_h'(w_h)(\psi),\eta \>  = \int_{\Omega} (\div (\text{cof} \, D^2 v_h) D \psi) \eta \ud x \\ &\qquad \qquad \qquad \qquad \qquad  - 
\int_{\Omega} (\div (\text{cof} \, D^2 w_h) D \psi) \eta \ud x \\
& = -\int_{\Omega} [(\text{cof} \, D^2 v_h) D \psi] \cdot D \eta \ud x \\ &\qquad \qquad \qquad + \int_{\Omega} [(\text{cof} \, D^2 w_h) D \psi] \cdot D \eta \ud x \\
& = \int_{\Omega} [(\text{cof} \, D^2 w_h - \text{cof} \, D^2 v_h) D \psi] \cdot D \eta \ud x.
\end{split}
\end{align}
By \eqref{meanv2} and an inverse estimate
\begin{align*}
|\<F_h'(v_h)(\psi),\eta\> - \<F_h'(w_h)(\psi),\eta \>| & \leq n(n-1)^2 ( ||v_h||_{2,\infty} \\
&   \quad \quad  + ||w_h||_{2,\infty})^{n-2} ||v_h-w_h||_{2,\infty}^{} ||\psi||_{1} ||\eta||_1 \\
||F_h'(v_h)(\psi) - F_h'(w_h)(\psi)|| & \leq n(n-1)^2 ( ||v_h||_{2,\infty} \\
&   \qquad \quad + ||w_h||_{2,\infty})^{n-2} ||v_h-w_h||_{2,\infty}^{} ||\psi||_{1} \\
& \leq n(n-1)^2C_{0} ( ||v_h||_{2,\infty}+ ||w_h||_{2,\infty})^{n-2} \\
&  \qquad \qquad \qquad   ||v_h-w_h||_{1}^{} ||\psi||_{1}. 
\end{align*}
In the case $n=3$ we have by an inverse estimate, the definition of $X^h$ and the assumption on $\delta$ \eqref{delta}
\begin{align*}
 ||v_h||_{2,\infty}+ ||w_h||_{2,\infty} & \leq ||v_h-Q_h u||_{2,\infty}+ ||w_h-Q_h u||_{2,\infty} + 2 ||Q_h u||_{2,\infty} \\
 & \leq C_{0} h^{-1-\frac{3}{2}} (||v_h-Q_h u||_{1}+ ||w_h-Q_h u||_{1}) \\
 & \qquad \qquad \qquad \qquad + C || u||_{2,\infty} \\
 & \leq \frac{C_{0}  \delta}{2} + C || u||_{2,\infty}\\
 & \leq \frac{C_{0} }{2} + C || u||_{2,\infty}.
\end{align*}
We conclude that \eqref{glc} holds. 
\end{proof}
We define
$$Y^h= \{ \, v_h \in V^h, v_h=g_h \, \text{on} \, \partial \Omega, ||v_h-u_h ||_1 <  \frac{\delta h^{1+\frac{n}{2}}}{4}\, \}.$$

\begin{rem}\label{Yh-rem}
For $v_h \in Y^h$, $ ||v_h-Q_h u ||_1 \leq  ||v_h-u_h ||_1+ ||Q_h u -u_h ||_1 < \delta h^{1+n/2}/2$ and hence $Y^h \subset X^h$.
\end{rem}
We will make the abuse of notation of denoting by both $u_k$ the solution of the iterative methods at both the continuous and discrete level. In the remainder of this paper, only discrete solutions are considered. This alleviates the notation.

Finally we note that we have the freedom to choose $\delta$ given by \eqref{delta} smaller. Indeed this will be necessary for the convergence of the pseudo transient continuation methods.

We make the assumption that \eqref{var1h} has a unique strictly convex solution. Recall from Remark \ref{conv-disc-sol} that this holds for example when \eqref{m1} has a smooth strictly convex solution and for $h$ sufficiently small.

\section{Convergence of the pseudo transient continuation methods } \label{continuation}
We now prove the convergence of the iterative methods \eqref{newtonseq}. 
The discretization of \eqref{newtonseq} depends on the choice of $L$: Given $\nu >0$ and a suitable initial guess, find $u_{k+1} \in V^h,
 u_{k+1} = g_h \,\text{on} \, \partial \Omega$ such that we have for all $\psi_h \in V_0^h$, when $L$ is the Laplace operator 
 \begin{align} 
\begin{split}
-\nu  \int_{\Omega}  (D u_{k+1}- D u_k) \cdot D \psi_h \ud x  & + \<F_h'(u_k) (u_{k+1}-u_k), \psi_h \> \\
& \qquad \qquad \qquad  = \<-(F_h(u_k)-f), \psi_h \>, \label{nvar12}
\end{split}
\end{align}
and when $L$ is the negative of the identity,
\begin{align} 
\begin{split}
-\nu  \int_{\Omega}  (u_{k+1}-u_k) \psi_h \ud x  & + \<F_h'(u_k) (u_{k+1}-u_k), \psi_h \> \\
& \qquad \qquad = \<-(F_h(u_k)-f), \psi_h \>.  \label{nvar11} 
\end{split}
\end{align}
We define
\begin{equation} \label{ch}
C_h = \frac{m}{C_1^2} + \frac{\nu h^2}{C_0^2},
\end{equation}
and require $h < \min \{\sqrt{2} C_0, C_0/C_1 \}.$
Thus since $h^2 < 2 C_0^2$ we have $2-h^2/C_0^2 > 0$ and if we require
\begin{equation} \label{nu-cond}
0< \nu < \frac{m}{C_1^2(2-\frac{h^2}{C_0^2})},
\end{equation}
we have
\begin{equation} \label{nu-cond2}
0< \nu < \frac{C_h}{2}.
\end{equation}

We now require that
\begin{equation} \label{delta3}
\delta < \min \bigg\{ \, 1, \frac{m'}{2 n C_{0}},\frac{2 C_h}{C_2}\, \bigg\}. 
\end{equation}

\begin{thm} \label{pseudo-cvg} 
Let $\Omega$ be convex with a Lipschitz continuous boundary and assume  that the spaces $V^h=S^1_d(\mathcal{T})$ have the optimal approximation property \eqref{schum} and satisfy the inverse estimates \eqref{inverse}.
A sequence defined by either \eqref{nvar12} or \eqref{nvar11} with a suitable initial guess 
and for $\nu, h$ sufficiently small converges to the unique strictly convex solution of \eqref{var1h}. Moreover
the convergence rate is linear.
\end{thm}
\begin{proof}
Define $\mathcal{M}_i: V_0^h \to (V_0^h)', i=1,2$ 
for $v, \psi_h \in V_0^h$ by
\begin{align*} \label{m-est}
\< \mathcal{M}_1(v),\psi_h \> = \int_{\Omega}  D v \cdot D \psi_h \ud x, \, \< \mathcal{M}_2(v),\psi_h \> = \int_{\Omega}  v \psi_h \ud x. 
\end{align*}
We note that
\begin{align}
||\mathcal{M}_i(v)|| \leq ||v||_1, v \in V_0^h, i=1,2.
\end{align}
Next, for $p \in V_0^h$, by \eqref{fundpos} and Poincare's inequality we have
$$
-(M + \nu) ||p||_1^2  \leq \<F_h'(v_h)(p), p\> - \nu |p|_1^2  \leq -\frac{1}{C_1^2}(m+\nu) ||p||_1^2.$$
Thus
\begin{equation}
|  \<F_h'(v_h)(p), p\> - \nu |p|_1^2 | =  -\<F_h'(v_h)(p), p\> + \nu |p|_1^2 \geq \frac{1}{C_1^2}(m+\nu)  ||p||_1^2.
\end{equation}
We have
\begin{align*}
||-\nu \mathcal{M}_1(p) + F'_h(v_h)(p)|| & = \text{sup}_{\psi_h \neq 0} \frac{|-\nu \mathcal{M}_1(p)(\psi_h) + F'_h(v_h)(p)(\psi_h)|}{||\psi_h||_1} \\
& \geq  \frac{|-\nu \mathcal{M}_1(p)(p) + F'_h(v_h)(p)(p)|}{||p||_1} \\
& =  \frac{|-\nu |p|_1^2 + F'_h(v_h)(p)(p)|}{||p||_1} \\
& \geq \frac{1}{C_1^2}(m+\nu)  ||p||_1.
\end{align*}
Similarly, since $C_0^{-2} h^2 ||p||_1^2 \leq ||p||_0^2 \leq ||p||_1^2$ 
$$
-(M + \nu) ||p||_1^2 \leq \<F_h'(v_h)(p), p\> - \nu ||p||_0^2  \leq - ( \frac{m}{C_1^2} + \frac{\nu h^2}{C_0^2}) ||p||_1^2.
$$
And again 
\begin{align*}
||-\nu \mathcal{M}_2(p) + F'_h(v_h)(p)|| & \geq   \frac{|-\nu ||p||_0^2 + F'_h(v_h)(p)(p)|}{||p||_1}  \\
&\geq ( \frac{m}{C_1^2} + \frac{\nu h^2}{C_0^2}) ||p||_1.
\end{align*}
For $h$ sufficiently small, i.e. $h \leq C_0/C_1$, we have $1/C_1^2 \geq h^2/C_0^2$ and 
we therefore have
\begin{align} \label{m-est2}
||p||_1 \leq \frac{1}{ \frac{m}{C_1^2} + \frac{\nu h^2}{C_0^2} } ||-\nu \mathcal{M}_i(p) + F'_h(v_h)(p)||, p \in V_0^h, i=1,2. 
\end{align}


We can now determine under which conditions when $u_k \in Y^h$ we have $u_{k+1} \in Y^h$ as well. Using \eqref{nvar12}, \eqref{nvar11}, $F_h(u_h)=f$ and the Mean Value Theorem,
\begin{align*}
-\nu \mathcal{M}_i(u_{k+1}-u_h)  +  F_h'(u_k)&  (u_{k+1}-u_h)  = -\nu \mathcal{M}_i(u_{k}-u_h) + F_h'(u_k) (u_k-u_{h}) \\
& \qquad \qquad \qquad - (F_h(u_k) -f)\\
& = -\nu \mathcal{M}_i(u_{k}-u_h) + F_h'(u_k) (u_k-u_{h}) \\
& \qquad \quad  - \int_0^1 F_h'(u_h + \theta (u_k - u_h))(u_k - u_h) \ud \theta\\
&  =\int_0^1 [F_h'(u_k) - F_h'(u_h + \theta (u_k - u_h))] 
(u_k - u_h)\\
& \qquad -\nu \mathcal{M}_i(u_{k}-u_h)  \ud \theta, 
\end{align*}
Using \eqref{ch}, \eqref{m-est2}, \eqref{m-est} and the generalized Lipschitz continuity property of $F'_h$, we get
\begin{align*}
C_h ||u_{k+1}-u_h||_1 & \leq \nu ||u_{k}-u_h||_1 +C_2 h^{-1-\frac{n}{2}}  ||u_{k}-u_h ||_1^2,
\end{align*}
and thus
\begin{align} \label{cvg-psd}
\begin{split}
||u_{k+1}-u_h||_1 & \leq \frac{\nu}{C_h} ||u_{k}-u_h||_1  +   \frac{C_2 h^{-1-\frac{n}{2}} }{C_h}  ||u_{k}-u_h ||_1^2,
\end{split}
\end{align}
By the definition of $Y^h$ and the choice of $\delta$, we have
$$
 \frac{C_2 h^{-1-\frac{n}{2}} }{C_h}  ||u_{k}-u_h ||_1 \leq \frac{\delta C_2}{4 C_h} < \frac{1}{2}.
$$
This gives by \eqref{nu-cond2} and \eqref{cvg-psd}
\begin{equation} \label{cvg-p}
 ||u_{k+1}-u_h||_1 <  ||u_{k}-u_h||_1,
\end{equation}
and we have proved that  $u_{k+1} \in Y^h$ when $u_{k} \in Y^h$.

We now assume that $u_0$ is chosen in $Y^h$. We have for $i=1,2$ 
\begin{align*}
\<F_h(u_{k+1}) -f, \psi_h\> & = \<F_h(u_{k+1}) - F_h(u_k) + F_h(u_k) -f,\psi_h \>\\
& =\<F_h(u_{k+1}) - F_h(u_k), \psi_h \> -  \< \nu \mathcal{M}_i(u_{k+1}-u_k),\psi_h \> \\
& \qquad \qquad \qquad \qquad \qquad \qquad - \<F_h'(u_k) (u_{k+1}-u_k), \psi_h \>\\
& = \< \int_0^1 [F_h'(u_k+ t (u_{k+1}-u_k)) -F_h'(u_k)] (u_{k+1}-u_k) \ud t \\
& \qquad \qquad \qquad \qquad \qquad \qquad \qquad -  \< \nu \mathcal{M}_i(u_{k+1}-u_k)
, \psi_h \>.
\end{align*}
We conclude from the general Lipschitz continuity property with $v_h=u_k$, an inverse estimate and  \eqref{m-est},
$$
||F_h(u_{k+1}) -f|| \leq C_2 h^{-1-\frac{n}{2}} ||u_{k+1}-u_k||_1^2 + \nu ||u_{k+1}-u_k||_1.
$$

Finally, by \eqref{m-est2} and the definition of the iterative methods \eqref{nvar12} and \eqref{nvar11}, 
$$
||u_{k+1}-u_k||_1 \leq\frac{1}{C_h} ||F_h(u_k) -f ||.
$$
We conclude that
$$
||F_h(u_{k+1}) -f|| \leq c_1(h) ||F_h(u_k) -f ||^2 +c_0(h) ||F_h(u_k) -f ||,
$$
for constants $c_0(h)=\nu/C_h$ and $c_1(h)$ which depends on $h$. By the assumption \eqref{nu-cond2} on $\nu$, we have
$c_0(h)< 1/2$.

Let $q=||F_h(u_0) -f ||$ and assume that $u_0$ is chosen so that $c_1(h) q < 1-c_0(h)$. We then have $$s\equiv c_1(h) q + c_0(h) < 1.$$ 
It follows that
$$
||F_h(u_{1}) -f|| \leq c_1(h) ||F_h(u_{0}) -f||^2 + c_0(h) ||F_h(u_{0}) -f|| =  s q, 
$$
and since $s < 1$,
\begin{align*}
||F_h(u_{2}) -f|| & \leq  c_1(h) ||F_h(u_{1}) -f||^2 + c_0(h) ||F_h(u_{1}) -f|| \\
& = ||F_h(u_{1}) -f|| (c_1(h) ||F_h(u_{1}) -f||^{} +c_0(h)) \\
& \leq ||F_h(u_{1}) -f|| (c_1(h) s q +  c_0(h)) \\
& \leq ||F_h(u_{1}) -f|| s \leq s^2 q.
\end{align*}
We conclude that $||F_h(u_{k}) -f|| \leq s^k q$. Using $F_h(u_h)=f$ and the Mean Value Theorem 
\begin{align*}
\< F_h(u_{k}) -f,  u_k-u_h \> &= \< F_h(u_{k}) - F_h(u_h),  u_k-u_h \> \\
& = \< \int_0^1 F'_h(t u_k + (1-t) u_h )(u_k-u_h)  \ud t ,  u_k-u_h\> \\
&  =\< \int_0^1 F'_h(t u_k + (1-t) u_h )(u_k-u_h),  u_k-u_h\>  \ud t.  \\
\end{align*}
Thus integrating \eqref{fundpos} with respect to $t$ 
we obtain
\begin{align*}
\frac{m}{C_1^2} ||u_k-u_h ||_1^2 & \leq \bigg|   \int_0^1 \< F'_h(t u_k + (1-t) u_h )(u_k-u_h),  u_k-u_h \> \ud t \bigg|\\
& = |\< F_h(u_{k}) -f,  u_k-u_h \>| \\
& \leq || F_h(u_{k}) -f|| \, ||u_k-u_h||_1.
\end{align*}
We conclude that
\begin{align*}
||u_k-u_h||_1 & \leq C ||F_h(u_{k}) -f|| \leq C q s^k,
\end{align*}
from which the convergence  follows. The convergence rate is given by \eqref{cvg-p}.
\end{proof}
\begin{rem}
The proof of convergence of the pseudo transient continuation methods also gives the convergence of Newton's method when $\nu=0$. In particular \eqref{cvg-psd} gives the quadratic convergence rate of Newton's method when $\nu=0$. The quadratic convergence rate of Newton's method was also proved in a more general context in \cite{Bohmer2008} where it is shown that the rate of convergence is independent of $h$, \cite{Bohmer2008} Theorem 9.1. The independence of the rate in terms of the mesh size is known as mesh independence principle.
\end{rem}

\begin{rem}
The introduction of the constant $C_h$ is motivated by our desire to have a unified analysis in \eqref{m-est2} for both types of pseudo transient continuation methods. Since
$$
||p||_1 \leq C ||-\nu \mathcal{M}_1(p) + F'_h(v_h)(p) ||, p \in V_0^h,
$$
we get for $\nu=0$ from \eqref{cvg-psd}
$$
||u_{k+1}-u_h||_1 \leq C h^{-1-\frac{n}{2}} ||u_{k}-u_h||_1^2,
$$
for a constant $C$ independent of $h$.

\end{rem}

\begin{rem}
The analysis above does not indicate whether \eqref{nvar12} should be preferred over \eqref{nvar11}. We view \eqref{nvar12} as a preconditioned version
of \eqref{nvar11}. Moreover, the numerical results indicate that the use of the Laplacian preconditioner improves the convexity property of the numerical solution. 
\end{rem}

\section{Convergence of the time marching methods} \label{time}

We now turn to the proof of one of the main results of this paper, the convergence analysis of the iterative method \eqref{m2} for the Monge-Amp\`ere
equation. 

Let $\nu =(M+m)/2$ and define a mapping $T_1: Y^h \to (V_0^h)'$ by 
\begin{equation}
\<T_1(v_h),\psi_h\> =  \int_{\Omega} D v_h \cdot D \psi_h \ud x + \frac{1}{\nu} \int_{\Omega} (\det D^2 v_h-f) \, \psi_h \ud x, 
\end{equation}
for $v_h \in Y^h, \psi_h \in V_0^h$.
The following lemma will make it possible to show that $T_1$ is a strict contraction.
\begin{lemma}\label{lem1cont} For $v_h \in Y^h$,
\begin{align*}
||T_1'(v_h)||_*\equiv  & \sup_{\psi_h \in V_0^h, \psi_h \neq 0} \frac{||T_1'(v_h)(\psi_h)||}{|\psi_h|_1} \\
&  \leq  
\sup_{\psi_h \in V_0^h, \psi_h \neq 0} \frac{|T_1'(v_h)(\psi_h)(\psi_h)|}{|\psi_h|_1^2} \leq \frac{M-m}{M+m}.
\end{align*}
\end{lemma}
\begin{proof}
Let $\alpha=\text{sup}_{\psi_h \in V_0^h, \psi_h \neq 0} \frac{|T_1'(v_h)(\psi_h)(\psi_h)|}{|\psi_h|_1^2}$. We have
\begin{equation} \label{ftip}
|T_1'(v_h)(\psi_h)(\psi_h)| \leq \alpha |\psi_h|_1^2, \psi_h \in V_0^h.
\end{equation}
 Since for $\mu_h \in V_0^h$, $||T_1'(v_h)(\mu_h)||=\text{sup}_{\eta_h \in V_0^h, \eta_h \neq 0} |T_1'(v_h)(\mu_h)(\eta_h)|/|\eta_h|_1$, we obtain
\begin{align*}
||T_1'(v_h)||_* = \text{sup}_{\mu_h, \eta_h \in V_0^h, \mu_h, \eta_h \neq 0} \frac{|T_1'(v_h)(\mu_h)(\eta_h)|}{|\mu_h|_1|\eta_h|_1}.
\end{align*}
But
\begin{align*}
T_1'(v_h)(\mu_h)(\eta_h) & = \int_{\Omega} D \mu_h \cdot D \eta_h \ud x -\frac{1}{\nu}
\int_{\Omega} [(\text{cof} \, D^2 v_h) D \mu_h] \cdot D \eta_h \, \ud x\\
& = \int_{\Omega} [(I -\frac{1}{\nu}(\text{cof} \, D^2 v_h)) D \mu_h] \cdot D \eta_h \, \ud x,
\end{align*}
where $I$ denotes the $n \times n$ identity matrix. 
Hence
\begin{align*}
\frac{T_1'(v_h)(\mu_h)(\eta_h)}{|\mu_h|_1|\eta_h|_1} = 
\int_{\Omega} [(I -\frac{1}{\nu}(\text{cof} \, D^2 v_h)) D \frac{\mu_h}{|\mu_h|_1}] \cdot D \frac{\eta_h}{|\eta_h|_1} \, \ud x.
\end{align*}
Next, we note that for fixed $v_h \in Y^h$, we can define a bilinear form on $V_0^h$ by the formula
\begin{align*}
(p,q) & = \int_{\Omega} [(I -\frac{1}{\nu}(\text{cof} \, D^2 v_h)) D p] \cdot D q \, \ud x.
\end{align*}
Then since
$$
(p,q) = \frac{1}{4} ((p+q,p+q) - (p-q,p-q)),
$$
we obtain
\begin{align*}
||T_1'(v_h)||_* &= \sup_{\mu_h, \eta_h \in V_0^h, \mu_h, \eta_h \neq 0} \frac{1}{4}
\bigg| \int_{\Omega} [(I -\frac{1}{\nu}\text{cof} \, D^2 v_h) \\
& \qquad \qquad \qquad D(\frac{\mu_h}{|\mu_h|_1}+ \frac{\eta_h}{|\eta_h|_1}] \cdot D (\frac{\mu_h}{|\mu_h|_1}+ \frac{\eta_h}{|\eta_h|_1}) \, \ud x \\
& \quad \quad  - \int_{\Omega} [(I -\frac{1}{\nu}\text{cof} \, D^2 v_h) D(\frac{\mu_h}{|\mu_h|_1}- \frac{\eta_h}{|\eta_h|_1}] \cdot 
D (\frac{\mu_h}{|\mu_h|_1}- \frac{\eta_h}{|\eta_h|_1}) \, \ud x \bigg|\\
\leq & \frac{\alpha}{4} \bigg(\bigg|\frac{\mu_h}{|\mu_h|_1}+ \frac{\eta_h}{|\eta_h|_1}\bigg|_1^2 + \bigg|\frac{\mu_h}{|\mu_h|_1}- \frac{\eta_h}{|\eta_h|_1}\bigg|_1^2 \bigg)= \alpha.
\end{align*}
By Lemma \ref{lem-1} we have
\begin{equation*}
(1-\frac{M}{\nu}) |w|_1^2 \leq \int_{\Omega} [( I -\frac{1}{\nu}(\text{cof} \, D^2 v_h)) D w] \cdot D w \, \ud x \leq  (1-\frac{m}{\nu}) |w|_1^2, w \in H_0^1(\Omega). 
\end{equation*}
Since $\nu =(M+m)/2$, $1-M/\nu= -(M-m)/(M+m)$ and $1-m/\nu=(M-m)/(M+m)$, we conclude that $\alpha \leq (M-m)/(M+m)$.
\end{proof}
We can now prove the following lemma
\begin{lemma}\label{lem2cont}
The mapping $T_1$ is a strict contraction in $Y^h$ with contraction constant $(M-m)/(M+m)$ for $\nu=(M+m)/2$.
\end{lemma}
\begin{proof}
Let $v_h$ and $w_h \in Y^h$. Then, using the Mean Value Theorem
\begin{align*}
||T_1(w_h)-T_1(v_h)|| & = ||\int_0^1 T_1'(v_h +t(w_h-v_h))(w_h-v_h) \ud t || \\
& \leq \int_0^1 ||T_1'(v_h +t(w_h-v_h))(w_h-v_h)|| \ud t.
\end{align*}
Since $w_h-v_h \in V_0^h$ and $v_h +t(w_h-v_h) \in Y^h, t \in [0,1]$, we obtain by Lemma \ref{lem1cont},
\begin{align*}
||T_1(w_h)-T_1(v_h)|| & \leq \int_0^1 \frac{M-m}{M+m} |w_h-v_h|_1 \ud t = \frac{M-m}{M+m} |w_h-v_h|_1.
\end{align*}
\end{proof}
\begin{rem}
For the operator $T_1$ to be a strict contraction, it is enough to have $\nu$ sufficiently large, i.e. $\nu >M$. 
In our computations, the value of $\nu$ is chosen ''adaptively'', i.e. we start with the value $\nu=50$ and if necessary we reduce or increase it for better accuracy. The situation is similar to the setting of adaptive mesh refinements where it is not known in advance where to do a local refinement and decisions are made based on computed results.
\end{rem}

\begin{rem}\label{remgrad}
By the inverse inequality, we have
\begin{equation*}
C_3 h^2 |w_h|_1^2 \leq ||w_h||_0^2 \leq C_1^2 |w_h|_1^2, w_h \in V_0^h. 
\end{equation*}
We may assume that $C_1 \leq 1$ by assuming that the domain is contained in a cube of side length at most 1.
It follows that for $w_h \in V_0^h$
\begin{equation*}
(C_3 h^2-\frac{M}{\nu}) |w_h|_1^2 \leq \int_{\Omega} w_h^2 \ud x- \int_{\Omega} [ \frac{1}{\nu}(\cof  D^2 v_h) D w_h] \cdot D w_h \, \ud x \leq  (1-\frac{m}{\nu}) |w_h|_1^2. 
\end{equation*}
As in the proofs of Lemmas \ref{lem1cont} and \ref{lem2cont}, we conclude that for $ \nu > M/(C_3 h^2)$, the mapping $T_2: Y^h \to (V_0^h)'$ defined by 
\begin{equation}
<T_2(v_h),\psi_h\> =  \int_{\Omega}  v_h  \psi_h \ud x + \frac{1}{\nu} \int_{\Omega} (\det D^2 v_h-f) \, \psi_h \ud x, 
\end{equation}
for $v_h \in Y^h, \psi_h \in V_0^h$ is a strict contraction. 

\end{rem}
We can now claim our main result, which is the convergence to $u_h$ of the sequence defined by $u_{k+1} \in V^h, u_{k+1}=g_h$ on $\partial \Omega$ and
\begin{equation} \label{iter1}
\nu \int_{\Omega} D u_{k+1} \cdot D \psi_h \ud x  = \nu \int_{\Omega} D u_{k} \cdot D \psi_h \ud x 
+ \int_{\Omega} (\det D^2 u_k-f) \, \psi_h \ud x,
\end{equation}
for $\psi_h \in V_0^h$.
\begin{thm} \label{time-cvg}
Let $\Omega$ be convex with a Lipschitz continuous boundary and assume  that the spaces $V^h=S^r_d(\mathcal{T})$ have the optimal approximation property \eqref{schum} and satisfy the inverse estimates \eqref{inverse}.
The sequence defined by \eqref{iter1} converges to the unique strictly convex solution $u_h$ of \eqref{var1h} for any initial guess $u_0$ in $Y^h$
and a suitable $\nu>0$ with a linear convergence rate.
\end{thm}
\begin{proof} 
The proof parallels Theorem 5.4 in \cite{Farago02}. 
Let us assume first that $u_k \in Y^h$. We have using \eqref{var1h}, or equivalently $\det D^2 u_h = f $ in $(V_0^h)'$,
\begin{align*}
\int_{\Omega} D (u_{k+1} - u_h) \cdot D \psi_h \ud x & =
\int_{\Omega} D (u_{k} - u_h) \cdot D \psi_h \ud x + \frac{1}{\nu} \int_{\Omega} \det D^2 u_k \, \psi_h \ud x \\
& \qquad \qquad \qquad \qquad - \frac{1}{\nu} \int_{\Omega} \det D^2 u_h \, \psi_h \ud x \\
& = \<T_1(u_k) - T_1(u_h),\psi_h \>.
\end{align*}
Taking $\psi_h=u_{k+1}-u_h$, we obtain
\begin{align*}
|u_{k+1} - u_h|_1^2 \leq ||T_1(u_k)-T_1(u_h)|| \, |u_{k+1} - u_h|_1 \leq \frac{M-m}{M+m} |u_{k} - u_h|_1 |u_{k+1} - u_h|_1,
\end{align*}
where for simplicity, we assume that the finite dimensional $V_0^h$ is equipped with the $| . |_1$ norm of $H_0^1(\Omega)$. We conclude that
$$
|u_{k+1} - u_h|_1 \leq \frac{M-m}{M+m} |u_{k} - u_h|_1.
$$
This also shows that if $u_{k} \in Y^h,$ then $u_{k+1} \in Y^h$ and concludes the proof.
\end{proof}
\begin{rem}It follows from the above result and Remark \ref{remgrad} that for a suitable initial guess and a suitable $\nu >0$, the sequence defined by
\begin{equation} \label{iter2}
\nu \int_{\Omega}  u_{k+1}  \psi_h \ud x  = \nu \int_{\Omega}  u_{k} \psi_h \ud x 
+ \int_{\Omega} (\det D^2 u_k-f) \, \psi_h \ud x,
\psi_h \in V_0^h, 
\end{equation}
with $u_{k+1} \in V^h, u_{k+1}=g_h \, \text{on} \, \partial \Omega$ also converges to the unique strictly convex solution of \ref{var1h}. Obviously the convergence
properties of \eqref{iter1} and \eqref{iter2} depend on the contraction constants of $T_1$ and $T_2$ respectively. Thus \eqref{iter1} is more robust than \eqref{iter2} in the sense that the choice of $\nu$ for \eqref{iter1} is less dependent on the discretization parameter $h$. As suggested in 
\cite{GlowinskiICIAM07} in the context of monotone schemes, the use of the Laplacian preconditioner results in a more efficient algorithm.

\end{rem}

\section{Numerical Results} \label{numerical}
The numerical results are obtained with the spline element method which we first review. We conclude the section with some heuristics about why our methods appear to enforce convexity.

\subsection{Spline element discretization} \label{spem}
We refer to \cite{Awanou2003,Awanou2005a,Awanou2006,Baramidze2006,Hu2007,Awanou2008} for a description of the spline element method.
We describe the method for linear problems and recall that the problems \eqref{m2} are linear problems.
Let $u \in V=H_0^m(\Omega), m \geq 1$ solve a variational problem $a(u,v)=f(v)$ with the conditions of the Lax-Milgram lemma satisfied. Take
$V_h$ as the spline space $S^r_d( \mathcal{T} )$ of smoothness $r$ and degree $d$, \eqref{sspaces}.
For $r=0$ and $d=1$ we have the space of piecewise linear continuous functions.

First, start with a representation of a piecewise discontinuous polynomial as a vector in $\mathbb{R}^N$, for some integer $N>0$.
Then express boundary conditions and constraints including global continuity or smoothness conditions as linear relations.
In our work, we use the Bernstein basis representation, \cite{Awanou2003,Awanou2008} which is very convenient to express smoothness conditions
and very popular in computer aided geometric design. Hence the term ``spline'' in the name of the method.
We can therefore identify the space $V_h$ with $\{c \in \R^N, R c = G \}$ for some integer $N$,
matrix $R$ and vector $G$. The discrete problem consists in finding
$c \in V_h, c^T K d = F^T d$ for all $d \in V_h$ for a suitable stiffness matrix $K$ and a load vector $F$. 
Introducing a Lagrange multiplier $\lambda$, 
the functional 
$$
K(c)d - L^Td + \lambda^T R d,
$$
vanishes identically on $V_h$. The stronger condition
$$
K(c) + \lambda^T R = L^T,
$$ 
along with the side condition $R c =G$ are the discrete equations to be solved.
We are lead to saddle point problems
\begin{align*}
\begin{split}
\left(\begin{array}{cc} K & R^T \\R & 0
\end{array} \right) \left(
\begin{array}{c} \mathbf{c} \\ \mathbf{\lambda}
\end{array}\right) = \left[
\begin{array}{c} F\\G
\end{array}\right]. \label{eqq4}
\end{split}
\end{align*}  
The ellipticity condition assures uniqueness of the component $c$
and the saddle point problems are solved by a version of the augmented Lagrangian algorithm
\begin{equation}\label{iter0}
(K +  \frac{1} { \mu} R^T R) c^{(l+1)} = K^T c^{(l)}  +  \frac{1} {
\mu} R^T G, \quad l=1,2,\ldots
\end{equation}
The convergence properties of the iterative method were given in \cite{Awanou2005}. Extensive implementation details can be found in \cite{Awanou2003,Awanou2006}.

\subsection{Numerical results}
For $n=2$, the computational domain is the unit square $[0,1]^2$
which is first divided into squares of side length $h$. Then each square is 
divided into two triangles by the diagonal with negative slope.
For $n=3$, the initial tetrahedral partition $\mathcal{I}_1$ consists in six tetrahedra. Each tetrahedron is then uniformly refined
into 8 subtetrahedra forming $\mathcal{I}_2$. In the tables, $n_{it}$ denotes the number of iterations. We refer to \cite{Awanou2003,Awanou2006} for implementation details of the method. All numerical experiments are with the versions of the iterative methods with Laplacian preconditioner. 

In general we did not try to choose the value of $\nu$ that would give the smallest number of iterations except in Tables \ref{tab-comp1} and \ref{tab-comp2} where we compare the performance of the two methods.

We use some standard test cases for numerical evidence for convergence to non smooth solutions of the elliptic Monge-Ampere equation.

Test 1: $u(x,y)=e^{(x^2+y^2)/2}$ so that 
$f(x,y)=(1+x^2+y^2)e^{(x^2+y^2)}$ and $g(x,y)=e^{(x^2+y^2)/2}$ on $\partial \Omega$.

Test 2: $u(x,y,z)=e^{(x^2+y^2+z^2)/3}$ so that
$f(x,y,z)=8/81(3+2(x^2+y^2+z^2)e^{(x^2+y^2+z^2)}$ and $g(x,y,z)=e^{(x^2+y^2+z^2)/3}$ on 
$\partial \Omega$. 

Barring roundoff errors, the methods introduced in this paper capture smooth solutions. For the two dimensional test function, Test 1, we give numerical results for successive refinements and for the three dimensional test function, we give numerical results for increasing values of the degree $d$ on two successive refinements.

\begin{table}
\begin{tabular}{|c|c|c|c|c|c|c|c|} \hline
$h$  & $n_{it}$ & $L^{2}$ norm & rate & $H^{1}$ norm & rate & $H^{2}$ norm & rate\\ \hline
$1/2^1$  & 236  & 4.1569 $10^{-6}$ & & 6.5142 $10^{-5}$ &  &1.9364 $10^{-3}$ &  \\ \hline
$1/2^2$  & 233 &  1.1504 $10^{-7}$ &5.17 & 2.3915 $10^{-6}$  & 4.77  &1.3444 $10^{-4}$& 3.85\\ \hline
$1/2^3$  & 233  &   3.2406 $10^{-9}$&5.15 &  8.4120 $10^{-8}$ & 4.83 &8.9366 $10^{-6}$ & 3.92 \\ \hline
$1/2^4$  & 233  &  4.5857 $10^{-10}$&2.82 &  4.7246 $10^{-9}$ & 4.15 & 6.0706 $10^{-7}$ & 3.88   \\ \hline
\end{tabular}
\bigskip
\caption{Time marching method for Test 1, $S^1_5$, $\nu=50$} \label{tab3}
\end{table}

\begin{table}
\begin{tabular}{|c|c|c|c|c|} \hline
d & $n_{it}$ &$L^{2}$ norm & $H^{1}$ norm &  $H^{2}$ norm \\ \hline
3 & 1&1.2338 $10^{-2}$ &  7.6984 $10^{-2}$ &    4.4411 $10^{-1}$ \\ \hline
4 & 270 &1.6289 $10^{-3}$ &  1.4719 $10^{-2}$ &    1.3983 $10^{-1}$ \\ \hline
5 & 135&1.5333 $10^{-3}$ &  8.7312 $10^{-3}$ &    6.0412 $10^{-2}$\\ \hline
6 & 424 &1.2491 $10^{-4}$ &  9.7458 $10^{-4}$ &    1.0473 $10^{-2}$ \\ \hline
Rate &  & 0.18 $0.25^{d-1}$ & 4.57 $0.25^d$ & 60.85 $0.3^{d+1}$\\ \hline
\end{tabular}
\bigskip
\caption{Time marching method for Test 2 (3D) on $ \mathcal{I}_1 $, $\nu=50$} \label{tab4}
\end{table}

\begin{table}
\begin{tabular}{|c|c|c|c|c|c|} \hline
d & $n_{it}$ &$L^{2}$ norm & $H^{1}$ norm &  $H^{2}$ norm \\ \hline
3 & 1&3.1739 $10^{-3}$ &  2.3005 $10^{-2}$ &   2.4496 $10^{-1}$ \\ \hline
4 &651 &3.2385 $10^{-4}$ &  3.5599 $10^{-3}$ &   5.2262 $10^{-2}$  \\ \hline
5 &744 &2.2730 $10^{-5}$ &  3.8977 $10^{-4}$ &   8.8978 $10^{-3}$  \\ \hline
6 &652 &1.1956 $10^{-6}$ &  2.2056 $10^{-5}$ &   6.0437 $10^{-4}$  \\ \hline
Rate &  & 0.72 $0.072^{d-1}$ & 29.44 $0.1^d$ & 861.43 $0.14^{d+1}$\\ \hline
\end{tabular}
\bigskip
\caption{Time marching method for Test 2 (3D) on $\mathcal{I}_2$, $\nu=50$} \label{tab5}
\end{table}

In the context of approximations by finite dimensional spaces, many finite element methods proposed, \cite{Dean2006,Feng2009b}, 
fail to fully capture the convexity of the solution
on the test case 

Test 3: $g(x,y)=0$ and $f(x,y)=1$. 

In Figure \ref{fig1} we give a plot of the graph of the solution as well as a section of the graph along the line $y=x$.

For the same test case, there is a concave solution. The concavity property of the concave solution obtained with the time marching method are better than the one obtained by the vanishing
moment methodology, \cite{Feng2009b}. This is illustrated in Figure \ref{fig2}.
\begin{figure}[tbp]
\begin{center}
\includegraphics[angle=0, height=6cm]{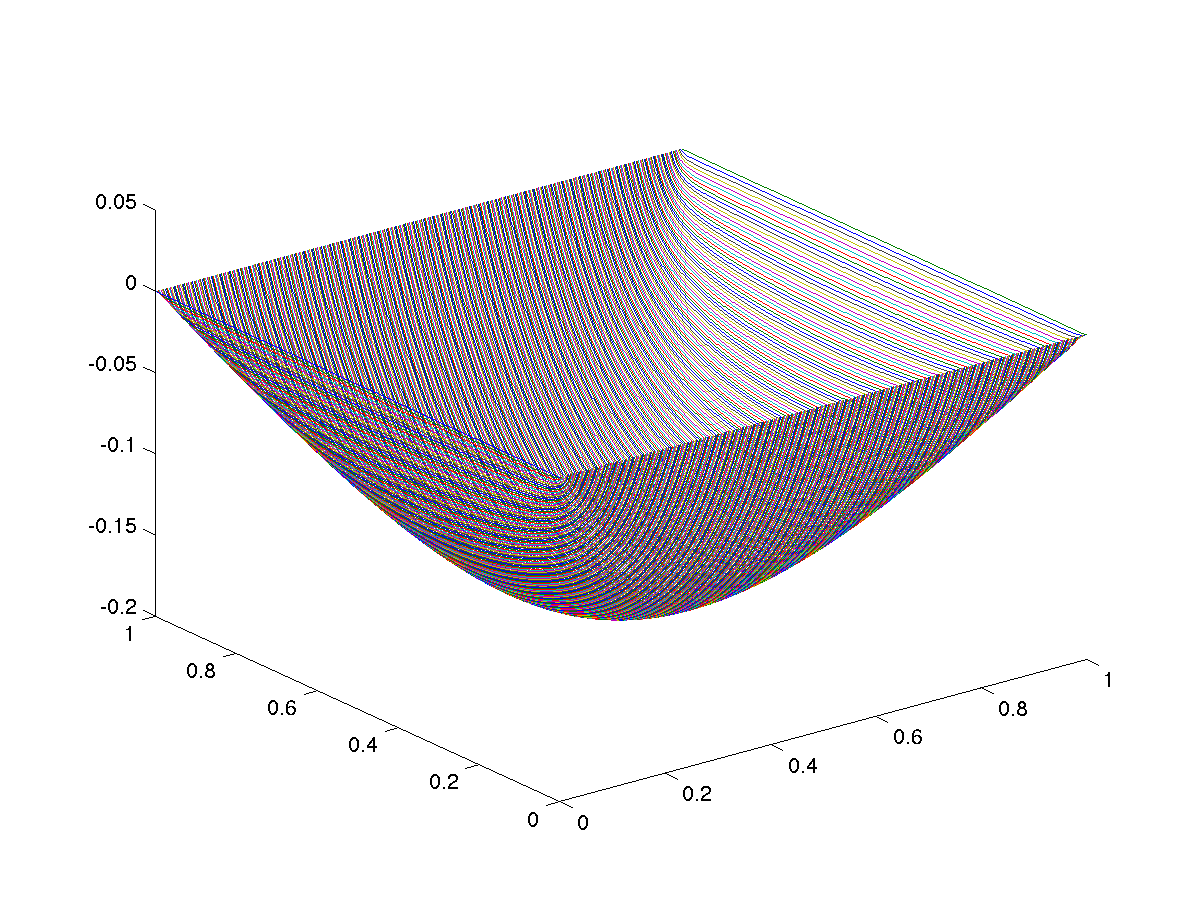}
\includegraphics[angle=0, height=7cm]{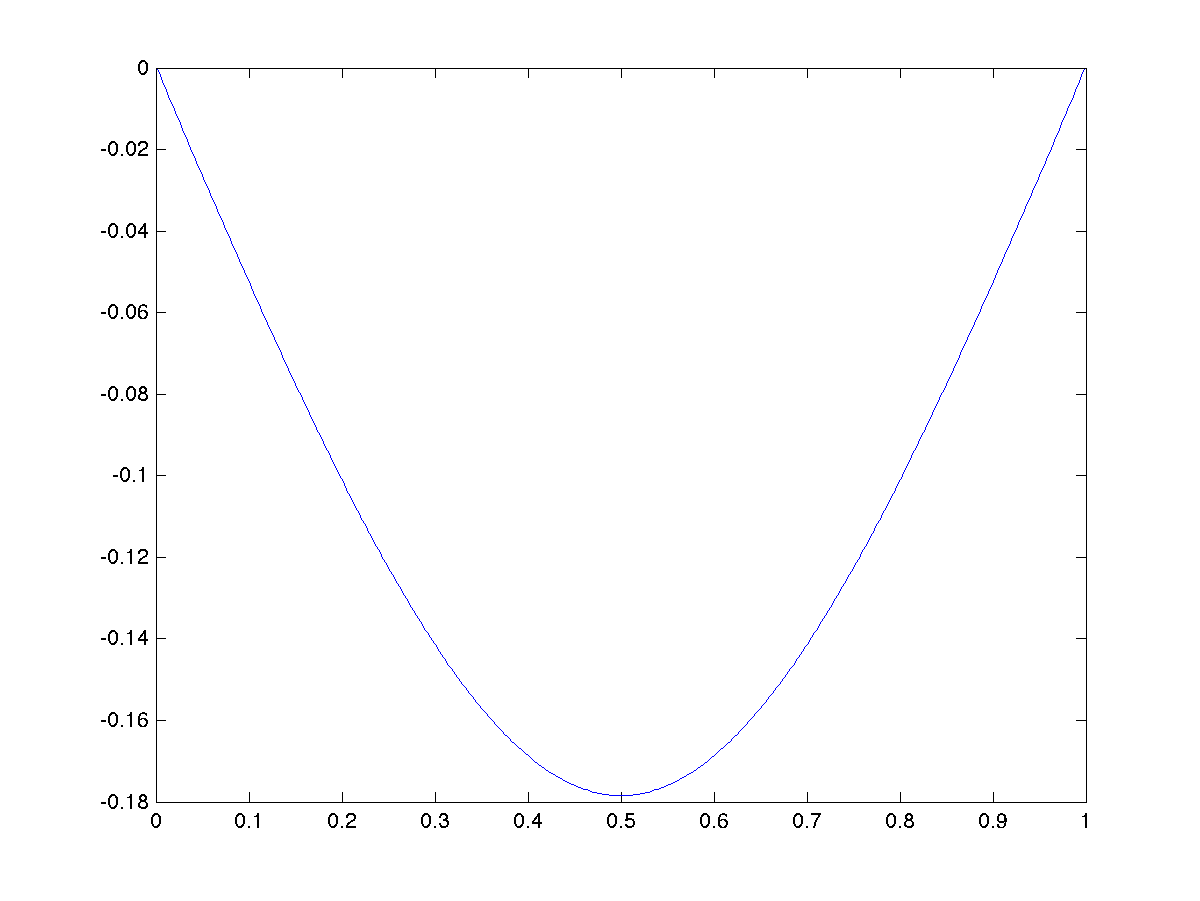}
\end{center}
\caption{Pseudo transient Test 3, convex solution: $h=1/2^4, d=5, \nu=7.5$.} \label{fig1}
\end{figure}
\begin{figure}[tbp]
\begin{center}
\includegraphics[angle=0, height=6cm]{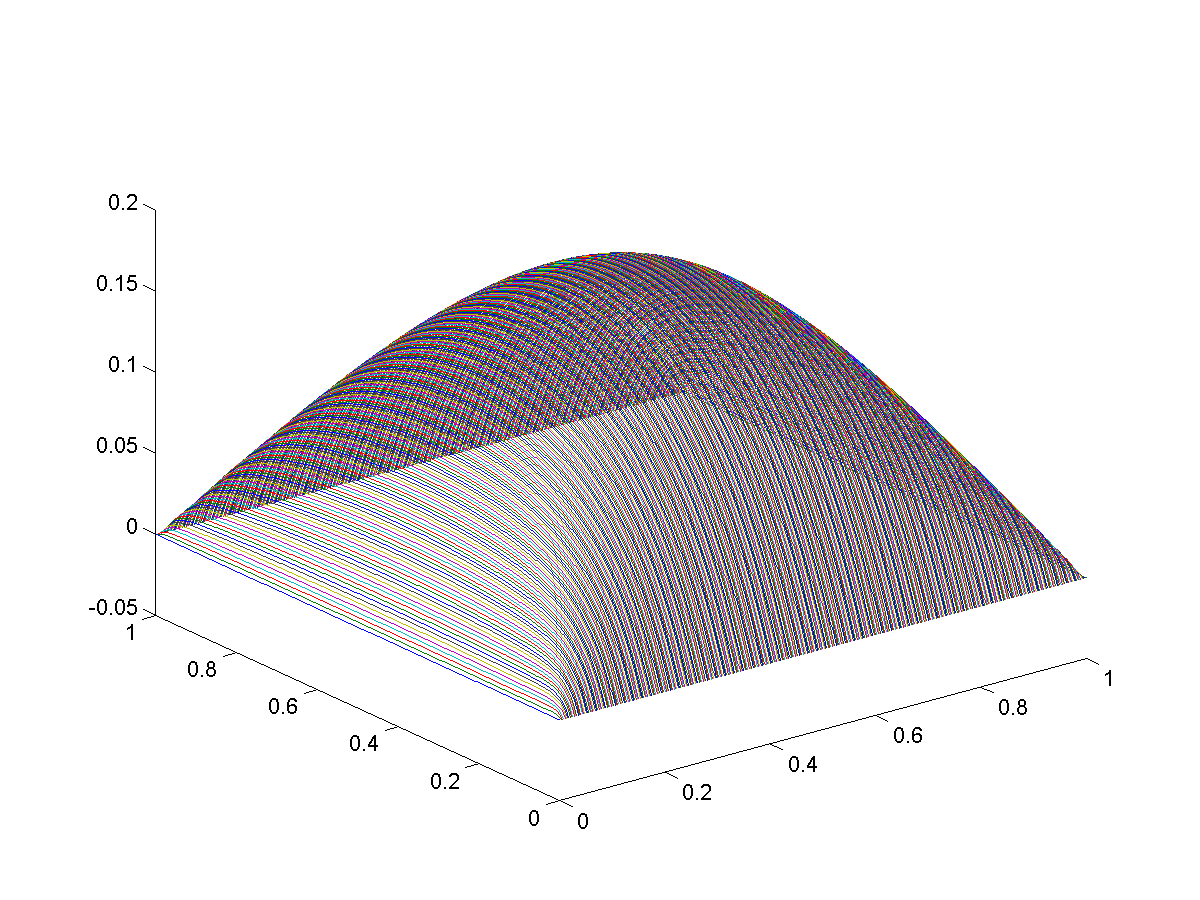}
\includegraphics[angle=0, height=7cm]{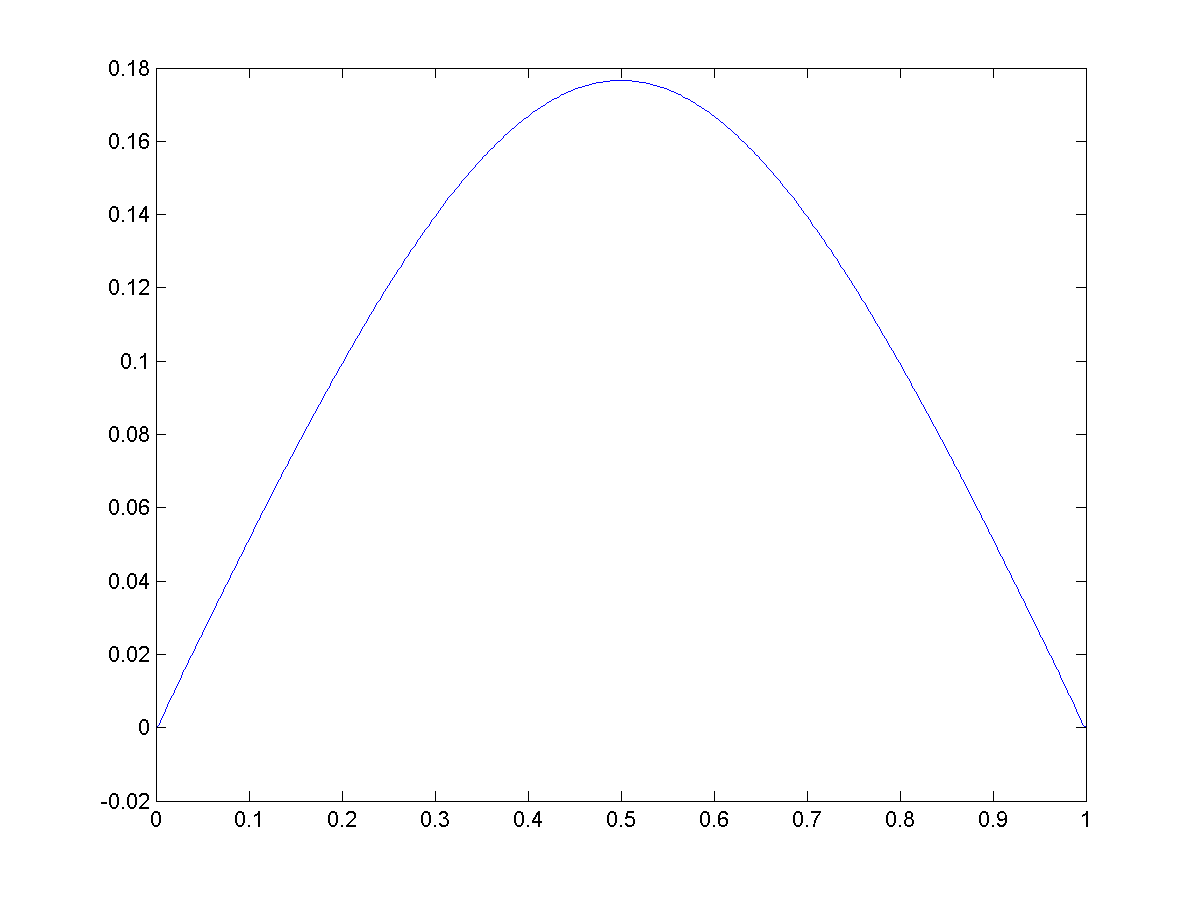}
\end{center}
\caption{Time marching, Test 3, concave solution: $h=1/2^4, d=5, \nu=50$.} \label{fig2}
\end{figure}


We now discuss how the two methods compare. First, we are solving the same discrete equations \eqref{var1h} by different iterative methods. Second, we noticed that the smaller $\nu$, the smaller the number of iterations. Thus for a smooth solution, the correct value of $\nu$ to take in the pseudo transient method is $\nu=0$ which is exactly Newton's method. In fact, Newton's method has been shown to have a quadratic convergence rate while the pseudo transient methods and time marching methods are shown in Theorems \ref{pseudo-cvg} and \ref{time-cvg} to have a linear convergence rate. Moreover the numerical errors of Tables \ref{tab3}, \ref{tab4} and \ref{tab5} are essentially the ones obtained with Newton's method as expected. We compare the performance of the methods on a non-smooth solution with known solution.  

Test 4:  $u(x,y)=-\sqrt{2-x^2-y^2}$ with corresponding $f$ and $g$.

\begin{table} 
\begin{tabular}{|c|c|c|c|c|c|} \hline
$h$  & $\nu$ & $n_{it}$ &    time & $L^2$ norm& rate \\ \hline
$1/2^1$   &  0 &   6 & $3.0328 10^{+0}$  & $2.1954 10^{-2}$ &\\ \hline
$1/2^2$  &  0   &  5  & $8.1365 10^{+0}$ & $3.6097 10^{-3}$& 2.60\\ \hline
$1/2^3$  & 0   &   6   &$3.8230 10^{+1}$ & $1.0685 10^{-3}$ & 1.76\\ \hline
$1/2^4$ & 3 &  56  & $1.5979 10^{+3}$       & $3.7666 10^{-4}$ & 1.50\\ \hline
\end{tabular}
\bigskip
\caption{Pseudo-transient method Test 4 $r=1, d=3$} \label{tab-comp1}
\end{table}

\begin{table}\label{tab-comp2}
\begin{tabular}{|c|c|c|c|c|c|} \hline
$h$  & $\nu$ & $n_{it}$   & time & $L^2$ norm& rate \\ \hline
$1/2^1$  & 2     & 35  & 6.2191 $10^{+0}$   &  $2.0721 10^{-2}$ & \\ \hline
$1/2^2$     &  2  & 89   &  6.0553 $10^{+1}$ & $1.8579 10^{-3}$ &3.48 \\ \hline
$1/2^3$    & 4.5   &  64    & 1.6849 $10^{+2}$ & $5.0438 10^{-4}$ &1.88 \\ \hline
$1/2^4$ & 11.5   & 151   & $1.7038 10^{+3}$      &  $2.1132 10^{-4}$& 1.25\\ \hline
\end{tabular}
\bigskip
\caption{Time marching method Test 4 $r=1, d=3$} \label{tab-comp2}
\end{table}
The time listed is in seconds and obtained on an imac running Mac OS 10.6.8 with a 2.4 Ghz intel core 2 duo and 4 GB of SDRAM memory. While for small values of $h$ the time marching method appears to take significantly more time, it is also significantly more accurate. For $h=1/2^4$ the time took by the two methods is almost the same with the time marching method giving a more accurate solution.

Next we consider  a non square domain. 

Test 5: we consider the unit circle discretized with a Delanauy triangulation with 824 triangles 
and $u(x,y)=x^2+y^2-1$ which vanishes on the boundary, Figure \ref{fig11}. 
\begin{figure}[tbp]
\begin{center}
\includegraphics[angle=0, height=4.5cm]{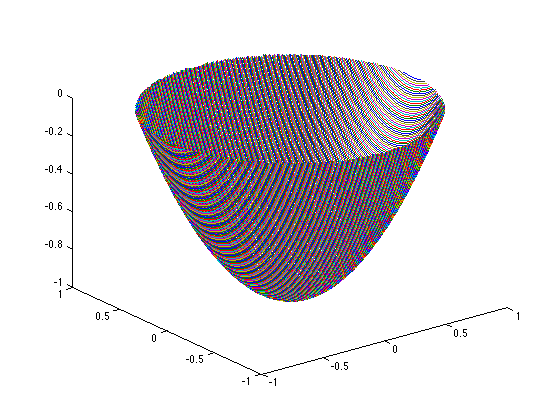}
\end{center}
\caption{$u(x,y)=x^2+y^2-1$ on a non square domain with pseudo transient $\nu=0, r=1, d=3$
} \label{fig11}
\end{figure}

We conclude this section with a test problem for a degenerate Monge-Amp\`ere equation

Test 6: $g(x,y)=|x-1/2|$ and $f(x,y)=0$. 

The graph of the function, Figure \ref{fig2} is singular along the line $x=1/2$. The approximations been $C^1$ do appear to capture the singularity but not the convexity of the solution. Somewhat better results are obtained with another iterative method discussed in an unpublished report available at 
http://arxiv.org/ abs/1012.1775. When the time marching method is discretized by the standard finite difference method the singularity is captured correctly. We wish to discuss these results in separate works.

\begin{figure}[tbp]
\begin{center}
\includegraphics[angle=0, height=4.5cm]{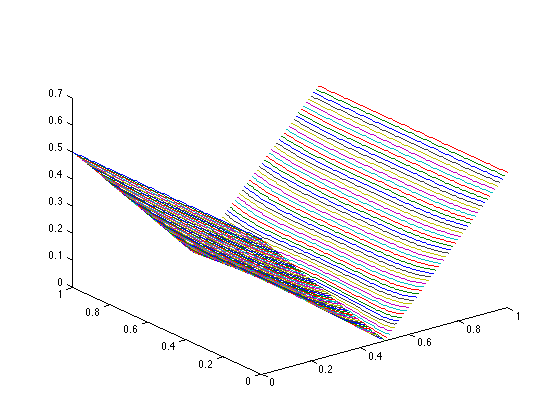} 
\includegraphics[angle=0, height=4.5cm]{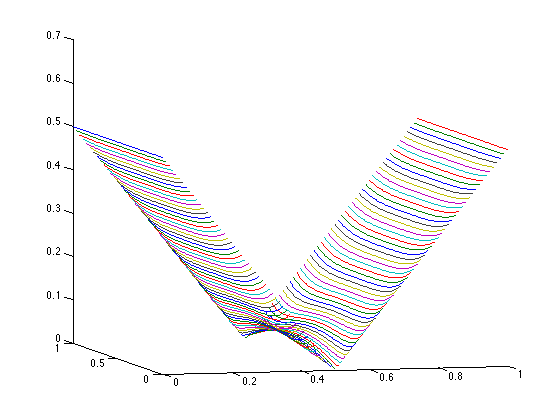}
\end{center}
\caption{$f(x,y)=0$ and $g(x,y)=|x-1/2|$ with time marching $\nu=50, r=1, d=5, h=1/2^4$
} \label{fig2}
\end{figure}

\subsection{Heuristics on convexity preservation}
When \eqref{m1} has a smooth strictly convex solution, Theorem \ref{errorest} establishes that the approximate solution is automatically convex. The numerical experiments indicate that in the non smooth case, discrete solutions are also convex. The result can be easily explained at the continuous level (for a smooth solution).

Assume that $f > c_0 > 0$ and that the sequence $u_k$ defined by
\begin{align} \label{pseudo-Delta}
\begin{split}
 \nu \Delta u_{k+1}  + (\text{cof} \ D^2 u_k) :D^2  u_{k+1} & = \nu \Delta u_{k}  + (\text{cof} \ D^2 u_k) :D^2  u_{k} \\
 & \qquad \qquad - \det D^2 u_k +f. 
 \end{split}
\end{align}
has been shown to converge to $u$ in the H$\ddot{\text{o}}$lder space $C^{2,\beta}(\Omega)$ for some $\beta$ in $(0,1)$. From the arithmetic-geometric inequality, we have 
\begin{align*}
 \frac{( \Delta u_{k} )^n}{n^n} \geq \det D^2 u_k.
\end{align*}
By the continuity of the eigenvalues, \eqref{cont-eig}, $\Delta v$ is bounded in a neighborhood of $u$ in which all $u_k$ belong for $k$ large enough. Choose $\nu$ such that $\nu \geq (n-1)(\Delta u_k)^{n-1}/  n^n$ for all $k$ and note that the right hand of \eqref{pseudo-Delta}
is equal to $\nu \Delta u_{k}  + (n-1)  \det D^2 u_k +f$. By the assumption on $\nu$, we get  $\nu \Delta u_{k+1}  + (\text{cof} \ D^2 u_k) :D^2  u_{k+1} \geq 0$. In the limit, we obtain 
$\nu \Delta u +  (\text{cof} \ D^2 u) :D^2  u_{} \geq 0$. Since $\det D^2 u \geq 0$ by assumption, we get $\Delta u \geq 0$.

As for the time marching method 
\begin{align*}
-\nu \Delta u_{k+1} = -\nu \Delta u_{k} + \det D^2 u_k -f, \ u_{k+1}=g \ \mathrm{on} \ \partial \Omega, 
\end{align*}
assume now again that  $f > c_0>  0$ and that the sequence $u_k$ has been shown to converge to $u$ in $C^{2,\beta}(\Omega)$ for some $\beta$ in $(0,1)$.
Choose $\nu$ such that $\nu \geq (\Delta u_k)^{n-1}/ n^n$. 
It follows from the arithmetic-geometric inequality that 
\begin{align*}
\nu \Delta u_k & \geq \frac{( \Delta u_{k} )^n}{n^n} \geq \det D^2 u_k.
\end{align*}
and so
 $-\nu \Delta u_{k} + \det D^2 u_k \leq 0$ and it follows that the time marching method also preserves the positivity of the Laplacian. 
 
In two dimensions  $\Delta u \geq 0$ and $\det D^2 u =f \geq 0$ imply that $D^2 u$ is positive.

\section*{Acknowledgements} The author would like to thank the referees for a careful reading of the paper and suggestions which led to a better presentation of the paper. 

The author acknowledges discussions with  F. Celiker,  B. Cockburn, W. Gangbo, 
R. Glowinski, M.J. Lai,
R. Nochetto, A. Oberman and A. Regev. 
The author was supported in part by
NSF grants DMS-0811052,  DMS-1319640 and the Sloan Foundation. 
This research was supported in part by the Institute for Mathematics and its Applications and the Mathematical Sciences Research Institute with funds provided 
by the National Science Foundation.

\end{document}